\documentclass[10pt,a4paper,oneside]{amsart}
\usepackage{amssymb,color,amsmath,amsfonts,pb-diagram,todonotes}
\usepackage[all]{xy}
\usepackage{pb-xy}
\usepackage[mathscr]{euscript}
\usepackage{url}
\usepackage[bookmarks]{hyperref}

\def\Z{\mathbb{Z}}
\def\Q{\mathbb{Q}}

\def\otimesover#1{\mathbin{\mathop{\otimes}_{#1}}}
\numberwithin{equation}{section}

\makeatother

\newtheorem{theorem}{Theorem}[section]

\newtheorem{lemma}[theorem]{Lemma}

\theoremstyle{definition}
\newtheorem{definition}[theorem]{Definition}

\newtheorem{remark}[theorem]{Remark}
\newtheorem{notation}[theorem]{Notation}

\newtheorem*{bs}{The Bieri-Strebel Theorem}
\newtheorem*{finitelygen}{Theorem~\ref{thm:finitelygen}}
\newtheorem*{finpres}{Theorem~\ref{thm:finpres}}
\newtheorem*{parafree}{Theorem~\ref{thm:parafree}}
\newtheorem*{finitelygenerated}{Theorem~\ref{thm:finitelygenerated}}
\newtheorem*{Gn}{Theorem~\ref{thm:Gn}}
\newtheorem*{tele4}{Theorem~\ref{thm:telescope}.(4)}

\newtheorem*{wrAfinite}{Theorem~\ref{thm:wreath}, (1) and (2)}
\newtheorem*{wrAinfinite}{Theorem~\ref{thm:wreath}, (3) and (4)}

\newtheorem{innercustomthm}{Theorem}
\newenvironment{customthm}[1]
    {\renewcommand\theinnercustomthm{#1}\innercustomthm}
  {\endinnercustomthm}

\def\g{\gamma}
\def\M{\mathscr{M}}
\def\R{{\bf R}}

\begin{document}

\title{Localization, metabelian groups, and the isomorphism problem}

\author{Gilbert Baumslag}
\address{CAISS and Dept. of Computer Science ,
City College of New York, Convent Avenue and 138th Street, New
York, NY 10031} \email{gilbert.baumslag@gmail.com}
\thanks{The research of the first author is supported by Grant CNS 111765 and the work done here
was initially caried out at IHES, whose hospitality is gratefully acknowleged}

\author{Roman Mikhailov}
\address{Chebyshev Laboratory, St. Petersburg State University, 14th Line, 29b,
Saint Petersburg, 199178 Russia and
St. Petersburg Department of the Steklov Mathematical Institute, Fontanka 27, Saint Petersburg, 191023 Russia.}
\email{rmikhailov@mail.ru}
\urladdr{http://chebyshev.spb.ru/roman-mikhailov}

\thanks{The research of the second author is supported by the Chebyshev Laboratory  (Department of Mathematics and Mechanics, St.
Petersburg State University)  under RF Government grant
11.G34.31.0026 and by JSC "Gazprom Neft", as well as by the  RF
Presidential  grant MD-381.2014.1.}

\author{Kent E. Orr}
\address{Dept. of Mathematics, Indiana University, Bloomington IN 47405}
\email{korr@indiana.edu}
\thanks{The third author thanks the
Simons Foundation, Grant 209082, for their support.}

\begin{abstract}
If $G$ and $H$ are  finitely generated, residually nilpotent
metabelian groups, $H$ is termed para-$G$ if there is a
homomorphism of $G$ into $H$ which induces an isomorphism between
the corresponding terms of their lower central quotient groups. We
prove that this is an equivalence relation. It is a much coarser
relation than isomorphism, our ultimate concern. It turns out that
many of the groups in a given equivalence class share various
properties including finite presentability. There are examples,
such as the lamplighter group,  where an equivalence class
consists of a single isomorphism class and others where this is
not the case. We  give several examples where we solve the
Isomorphism Problem. We prove also that the sequence of torsion-free ranks of
 the lower central quotients of  a  finitely
generated metabelian
group is computable. In a future paper we plan on proving that there is
an algorithm to compute the numerator and denominator of  the
rational Poincar\'e series of a finitely generated metabelian group and will carry out this
computation in a number of examples,  which may shed a tiny bit of light on the
Isomorphism Problem.  Our proofs use
localization,  class field theory and some constructive commutative algebra.

\end{abstract}

\maketitle

\section{Introduction}\label{section:intro}
\subsection{Preliminary remarks}
In a recent paper \cite{BMO1}, entitled ``A new look at finitely
generated metabelian groups'',  we outlined  a number of ideas for
exploring finitely generated metabelian groups. These ideas arise
from several seemingly different sources - algebraic geometry,
algebraic number theory, combinatorial group theory and
constructive commutative algebra. Here we will provide  some of
the details  briefly sketched in that paper and discuss and describe
some additional theorems that our first paper has given rise to.

We have chosen in
this paper to take a purely combinatorial
 view of our work. In the third
paper of this series we will take a more homological, functorial
approach using localization of groups, providing an alternative approach
to the material which we hope will lead to further understanding.

\subsection{Finitely generated metabelian groups}
A group $G$ is metabelian if its derived group, $A$, is abelian. In the first in a series of
fundamental papers going back to 1954, Philip Hall \cite{Hall1} observed that $A$ can be viewed as a
module over the integral group ring $\Z[H]$ of the factor derived group $H=G/A$, where $H$ acts on $A$
by conjugation. In the event that $G$ is finitely generated, $H$ is a finitely generated abelian group and,
as Hall noted, $A$ is  a finitely generated module over the finitely generated commutative ring $\Z[H]$.
It follows from Hilbert's basis theorem that $G$ satisfies the maximal condition for normal subgroups
and hence that the number of isomorphism classes of finitely generated metabelian groups is countable. As a
consequence every finitely generated metabelian group, viewed in the category of metabelian groups, has a finite description termed a {\em preferred presentation}  in  ~\cite{BCR} (see also~\cite{Lennox-Robinson}) defined
as follows: \\

\noindent A  {\em preferred presentation} of a finitely generated metabelian group $G$ is a presentation
which takes the form

\[
G=\langle g_1,\dots,g_t\ |\ R_1\cup R_2\rangle
\]
where
\begin{enumerate}
\item $R_1$ is a finite set of words of the form
\[
w=\prod_{1\leq i\leq j \leq   t}  [g_i,g_j]^{u_{ij}}\,
\]
and we use the usual notation $[x,y]$ for $x^{-1}y^{-1}xy$, $y^x$ for $x^{-1}yx$ and the $u_{ij}$ are words of the form
$g_1^{m_1}\dots g_t^{m_t}$;
\item $R_2$ is a finite set of words $uw$ where  $u$ has the form $g_1^{m_1}\dots g_t^{m_t}$ and
\[
w=\prod_{1\leq i\leq j\leq t }[g_i,g_j]^{v_{ij}}\, ,
\]
with $v_{ij}$ of the form $g_1^{n_1}\dots g_t^{n_t}$.\\
\end{enumerate}

\noindent Thus the words in $R_1$,  together with the addition of all commutators  $[x,y]$ where $x$ and $y$ take
the form $[g_i,g_j]^{u_{ij}}$, are a regular  presentation of $A$ while those in $R_2$ read modulo $A$,
yield a finite presentation of $G/A$. It follows that there is a recursive enumeration of all (preferred) presentations
of finitely generated metabelian groups. These presentations are the finite descriptions  needed in any discussion
of algorithms involving finitely generated metabelian groups.

\subsection{Algorithmic problems about finitely generated metabelian groups}
Both the word and conjugacy problems about finitely generated metabelian groups
are solvable.  The first is essentially due to Hall~\cite{Hall1} and the second to Noskov \cite{Noskov}.
The Isomorphism Problem
however remains almost untouched, except for the theorem of Groves and Miller~\cite{GrovesMiller}:
{\it The Isomorphism Problem for finitely generated free metabelian groups, i.e., the factor groups of
finitely generated free groups by their second derived groups, is solvable.}
In other words if we write down a preferred presentation of a finitely generated free metabelian
group $G$  and recursively enumerate all preferred presentations of finitely generated metabelian groups,
then the subset consisting of those presentations of groups isomorphic to $G$ is recursive. Another  way of putting
this is that there is an algorithm to decide whether or not any preferred presentation defines a group isomorphic
to $G$.

There are many other solvable algorithmic problems  about finitely generated metabelian groups:

\begin{enumerate}
\item There is an algorithm to find a preferred presentation of a finitely generated subgroup of a finitely
generated metabelian group.

\item There is an algorithm to determine if a finitely generated metabelian group is torsion-free.

\item There is an algorithm to find a module presentation of the derived group of a finitely generated metabelian group.

\item There is an algorithm to decide if a finitely generated metabelian group is residually nilpotent.

\item There is an algorithm to find the center of a finitely generated metabelian group.

\item There is an algorithm to find the centralizer of a finitely generated subgroup in a finitely generated metabelian group.

\end{enumerate}
We recommend Lennox and Robinson~\cite{Lennox-Robinson} as an excellent general reference as well as Baumslag, Cannonito, and Robinson~\cite{BCR} and the papers of Seidenberg~\cite{Sei1} and~\cite{Sei2}.  The specific algorithms above can be found, respectively, in~\cite[Page 185]{Lennox-Robinson},~\cite[Cors. 4.4, 3.1, 9.2, Thms.~ 3.5, 6.1]{BCR}

\subsection{The Isomorphism Problem}
Remeslennikov~\cite{Remeslennikov} has proved that every finitely
generated metabelian group has a verbal subgroup of finite index which is residually
nilpotent. So the Isomorphism Problem for finitely generated metabelian groups can be
broken down into the Isomorphism Problem for finitely generated residually nilpotent metabelian
groups, the study of the finite metabelian subgroups of their automorphism groups and the  study
of the finite extensions that come into play. This suggests that one approach to
the Isomorphism Problem is to focus attention on the class $\M$ of finitely generated, residually nilpotent
metabelian groups.

\begin{definition}  We term a group in the class $\M$ of finitely generated, residually nilpotent, metabelian groups  an {\em $\M$-group.}
\end{definition}

Since the preferred presentations
of such $\M$-groups are not finite, in order to take advantage of the fact that they are residually nilpotent
and the rich algorithmic properties of finitely presented nilpotent groups, we need to prove that
 the lower central sequences of $\M$-groups are computable. Hence the lower central sequences of $\M$-groups provide us with a computable set of invariants of $\M$-groups. This suggests that we focus on the following
\begin{definition}
Two groups $G$ and $H$ have the {\em same lower central sequences} if $G/\gamma_n(G) \cong H/\gamma_n(H)$ for every $n$.
\end{definition}
\noindent Here  $\gamma_n(X)$ denotes the
$n^{th}$ term of the lower central series of  the group $X$.

Some of the properties of $\M$-groups can be encapsulated in a construction introduced by
J. P. Levine \cite{Levine1}. Levine called this an {\em algebraic closure of the group.}  We restrict Levine's group closure to the class of $\M$-groups, and use an alternate but equivalent definition in this case.  We call this special case of Levine's group closure the {\em Telescope of the group $G$}, as this name suggests the structure of the group closure when considering $\M$-groups.  Nonetheless, Levine was aware of the telescoping structure of his group closure, at least for those  $\M$-groups which are semi-direct products.  The Telescope and it's applications will be discussed in detail in \S\ref{section:main} where we will formulate some of our results, some of which were reported in our previous paper \cite{BMO1}.
We also consider the better known and more studied pro-nilpotent completion of a group $G$, which we denote $\widehat{G}$. The Telescope of $G$ lies within $\widehat{G}$. In the event that the abelianization of $G$ is finitely generated (in particular for $\M$-groups), $\widehat{G}$ has the same lower central sequences as $G$~\cite[Theorem~13.3]{Bousfield}. If a residually nilpotent group $G$ is not nilpotent, $\widehat{G}$ is uncountable. Nonetheless in the case where $G$ is an $\M$-group, we shall prove that $\widehat{G}$ does share an important and interesting property with $G$. We shall also discuss this further in \S\ref{section:main}.

\subsection{Acknowledgements}  We thank James F. Davis for helpful discussions regarding number theory.

\section{Our main results}\label{section:main}
Recall that two groups $G$ and $H$ have the same lower central sequences if there is a sequence of isomorphisms
\[
\phi_n: G/\gamma_n(G) \longrightarrow H/\gamma_n(H)\, (n=1,2,\dots, n, \dots)
\]
between their lower central quotient groups. We have not been able to prove that there is
a homomorphism $\phi:G \longrightarrow H$ which induces such a sequence of isomorphisms
between their lower central sequences.

\begin{definition} Let $G$ and $H$ be groups.  We say that
$H$ is para-$G$ if there is a homomorphism of $G$ into $H$ which induces isomorphisms between the
corresponding quotients of  their lower central series.
\end{definition}
 It is this relationship
that we will  explore in some detail here.

\subsection{Para-equivalence}
We begin with  the following  theorem:

\begin{finitelygen}
Let $G$ and $H$ be residually nilpotent metabelian groups.
If $H$ is finitely generated and if $H$ is para-$G$ then $G$ is also finitely generated.

That is, there exists a para-$G$ $\M$-group only if $G$ is an $\M$-group.
\end{finitelygen}

\noindent Theorem~\ref{thm:finitelygen} is an  indication  that there is a connection between groups with the same lower central
sequences and their structure. A more important connection is contained in
W. Magnus'  fundamental paper  \cite{Magnus} in 1935 where he proved the following

\begin{theorem}[Magnus]\label{thm:Magnus}
If $\phi$ is a homomorphism of a residually nilpotent group $G$ into a group $H$ which induces  isomorphisms
between the respective terms of their lower central sequences, then $\phi$ is a monomorphism.
\end{theorem}

The relation of being para-$G$ is much stronger than one might suspect, as one sees from the following theorem:

\begin{tele4}
 Let $G$ and $H$ be $\M$-groups. Then $G$ is para-$H$ if and only
 if  $H$ is para-$G$.
 \end{tele4}

The above theorem is part of our main result, our Telescope Theorem, Theorem~\ref{thm:telescope},
 which we will discuss and prove in \S\ref{section:telescopeproof}.

It follows that
this property, para-$G$, is an equivalence relation on the class $\M$.

\begin{definition} If $H$ is para-$G$ and $G$ is para-$H$ we say $G$ and $H$ are {\em para-equivalent.}
\end{definition}

That is, our Telescope Theorem implies that the relation of being para-$G$ is an equivalence relation when restricted to $\M$-groups.

We will construct in \S\ref{section:examples} examples
 of para-equivalence classes for $\M$-groups which consist of more than a single isomorphism class. Notice that Theorem~\ref{thm:Magnus} and~\ref{thm:telescope}.(4) together imply that if the $\M$-groups $G$ and $H$ are para-equivalent
 then $G$ is isomorphic to a subgroup of $H$ and $H$ is isomorphic to a subgroup of
 $G$. So para-equivalence can be viewed
as a coarse form of isomorphism and can be compared with isoclinism, an equivalence
relation that Philip Hall introduced in an attempt to classify finite p-groups.   Another
consequence of the Telescope Theorem is the following

\begin{finpres}\label{thm:finpres}
Suppose that two $\M$-groups are para-equivalent. Then either both of them are finitely presented
or neither of them is.
\end{finpres}

The conclusion of Theorem 6.1 does not hold without restriction, since
Bridson and Reid~\cite{BR} have recently constructed examples of finitely generated residually nilpotent groups with the same lower central sequences
and with the following properties:
\begin{enumerate}
\item There is a homomorphism of one of the groups into another which induces isomorphisms between their
lower central sequences;
\item one of the groups is finitely presented whereas the other is not;
\item one of the groups has finitely generated second homology group whereas the other does not.
\end{enumerate}
Their work follows on our
ongoing and earlier work~\cite{BMO1} and~\cite{BR} where a closer connection is established between certain groups with the same lower central
sequences.

Another  interesting consequence of  Theorem~\ref{thm:telescope}.(4) involves the subgroups of free metabelian groups:

\begin{parafree}
Let $F$ be a finitely generated free metabelian group and suppose that $G$ is  an $\M$-group with the same lower central sequences as $F$. Then $G$ is isomorphic to a subgroup of $F$.
\end{parafree}

Since there exist a wide variety of $\M$-groups with the same lower central sequences as a free
metabelian group, it follows that the subgroup structure of finitely generated free metabelian groups is extremely complicated.

\subsection{Pro-nilpotent completions}
It is not hard to see that
if a residually nilpotent group is torsion-free,
then so too is its pro-nilpotent
 completion. As previously mentioned, the pro-nilpotent completion of an $\M$-group,
$G$, has the same lower
central sequence as  $G$. It would be interesting
to explore other properties of a residually
nilpotent groups which are inherited by their pro-nilpotent
completions. We give one example of a theorem of this kind:

\begin{finitelygenerated}
Let the $\M$-group $G$ be polycyclic.
Then $\widehat{G}$ is locally polycyclic, i.e., its finitely generated
subgroups are polycyclic.
\end{finitelygenerated}

\noindent Although we have not done so, in all likelihood our proof
of Theorem~\ref{thm:finitelygenerated} will carry
over to polycyclic groups in general.

\subsection{The Isomorphism Problem and Poincar\'e series}
 We recall that if  $G$ is a group and  $r_n=r_n(G)$ is the torsion-free rank of
 $\gamma_n(G)/\gamma_{n+1}(G)$.
then we term
\[
P(G)=\sum_{n=1}^{\infty}r_nx^n
\]
the {\em rational Poincar\'e series} of $G$. Such series have generally been
studied in connection with graded
modules over graded commutative rings. The Poincar\'e series of graded modules are rational functions, see e.g.,
Atiyah and MacDonald \cite{AtiyahMacDonald}. Baumslag first defined and investigated the Poincar\'e series of finitely generated residually torsion-free nilpotent metabelian groups and showed that this is a rational function in~\cite{Baumslag4} (see  also Groves and Wilson \cite{GrovesWilson}).
As a small contribution to the isomorphism  problem for finitely generated metabelian groups, we add parts $(2)$ and $(3)$ of the theorem below to the first author's prior result $(1)$, where $(3)$ can be viewed as an addendum to the theorem of Groves and Miller~\cite{GrovesMiller}:
\begin{customthm}{7.3}\label{poincare}
Let $G$  be a finitely generated metabelian  group.
Then  the following hold:
\begin{enumerate}
\item The rational Poincar\'e series  $P(G)$ of $G$ is a rational function.
\item There is an algorithm to compute the series $P(G)$.
\item The quotient of the two polynomials which is the Poincar\'e series of a free metabelian group
expressed as rational function is computable.
\end{enumerate}
\end{customthm}

\subsection{More on the Isomorphism Problem}
As already noted, Theorem~\ref{thm:telescope} (4),  provides a coarse classification called para-equivalence
 of certain finitely generated
 metabelian groups.  It points
the way to an approach to  the Isomorphism Problem and
gives rise to a number of questions.
For example if the para-equivalence class of a given group is a singleton,
can one solve the Isomorphism Problem for that group? In \S\ref{section:examples} we will
construct some examples of  groups which lie
in a single equivalence class and solve the Isomorphism Problem for
these groups. Furthermore,
even if the equivalence class
of a given group is not a singleton, we shall show that the Isomorphism Problem is
sometimes solvable for such a group. We
record here some samples of these kinds of results, which we will discuss
more fully in \S\ref{section:examples}. We also gather together in \S\ref{section:singletons} some results and
approaches to the Isomorphism Problem
which make use of ideal theory and class field theory to distinguish
some metabelian groups from
one another.  Although Grunewald and Segal \cite{Segal2} have solved the Isomorphism Problem for
polycyclic groups,
this approach is of independent interest and will be discussed
in a further paper in this series.

\subsection{Examples}
First we give two families of $\M$-groups where each para-equivalence class contains a unique group up to isomorphism, and for which the Isomorphism Problem is solvable.

\begin{Gn}
Let $G_n=\langle a,t\ |\ t^{-1}at=a^n\rangle,\  n \neq 2.$ Then
the following hold:
\begin{enumerate}
\item $G_n$ is residually nilpotent;
\item any $\M$-group with the same lower central sequence as  $G_n$ is isomorphic to $G_n$;
\item the Isomorphism Problem is solvable for each $G_n$.
\end{enumerate}
\end{Gn}

In order to formulate our next two theorems, we use wreath products, which we
will define in \S\ref{section:examples}.  These two theorems are combined in \S\ref{section:examples} into Theorem~\ref{thm:wreath}.

\begin{wrAfinite}\label{thm:uniquewreath}
Let $W$ be the wreath product of an abelian group $A$ of prime order and an infinite cyclic group.Then
\begin{enumerate}
\item $W$ is residually nilpotent;
\item any $\M$-group with the same
lower central sequences as $W$ is isomorphic to $W$;
\item the Isomorphism Problem is solvable for $W$.
\end{enumerate}
\end{wrAfinite}
 
It is worth noting that this theorem includes the lamplighter group,
the wreath product of a group of order
two and an infinite cyclic group. Theorem~\ref{thm:wreath}, (1) and (2), 
is capable of  considerable generalization; however
because this will take us too far astray from our current work,
we prefer to leave its formulation and proof  to another time.

Our next example involves the wreath product $W=\langle
a\rangle\wr \langle t\rangle$ of the infinite cyclic group $A$  on
$a$ by the infinite cyclic group $T$ on $t$. So $W$ is generated
by $a$ and $t$, the conjugates of $a$ by the different powers of
$t$ freely generate a free abelian group $B$ and $W=B \rtimes T$.
We then have the following

\begin{wrAinfinite}\label{thm:ideals}
Let $W=A  \wr T$. Then
\begin{enumerate}
\item $W$ is  residually-torsion free nilpotent;
\item the subgroup $V$ of $W$ generated by $a^{-1}(a^2)^t$, $a^2a^{-t}$, and $t$ has the same
lower central sequences as $W$;
\item $W \ncong V$.
\item the Isomorphism Problem is solvable for $W$.
\end{enumerate}
\end{wrAinfinite}

As a further example we will construct two polycyclic, residually nilpotent,
metabelian $\M$-groups with the same lower central
sequences which are not isomorphic.
We will not describe these
groups here. They are constructed in \S\ref{section:examples} as infinite cyclic extensions
of the additive groups of rings
of integers of algebraic number fields with class numbers at least two.
The argument used to prove
that the groups involved are not isomorphic require that
certain ideals are not principal, an idea used in part (iii) of the proof of Theorem~\ref{thm:uniquewreath} as well.

\section{The arrangement of  the rest of  this paper}\label{section:arrange}

Before describing the main contents of this paper we will list some of the notation and
standard definitions that will be used throughout.

\subsection{Definitions and notation}
Let $G$ be a group and let $x_1, x_2, \dots$ be elements of $G$.
We denote the commutator $x_1^{-1}x_2^{-1}x_1x_2$  by  $[x_1,x_2]$
and define, for $n>0$,
\[
[x_1,\cdots,x_{n+1}] = [[x_1,\cdots,x_n],x_{n+1}].
\]
If  $H$ and $K$ are subgroups of $G$, we define
\[
[H,K]=gp([h,k]\mid h \in H, k \in K)
\]
Its subgroup $[G,G]$ is termed the {\em derived group}. $G$ is
{\em metabelian} if its derived group is  abelian. The {\it lower
central series}
\[
 G = \gamma_1(G) \geq \gamma_2(G) \cdots
 \]
of $G$ is defined inductively by setting
\[
\gamma_{n+1}(G) = [\gamma_n(G),G]
\]
and the sequence
\[
G/\gamma_2(G), G/\gamma_3(G),\dots, G/\gamma_n(G),\dots
\]
is called {\em the lower central sequence of} $G$. A group
 $G$ is {\em residually nilpotent} if
\[
\bigcap_{n=1}^\infty\gamma_n(G) = 1.
\]

The rest of the paper is arranged as follows.

\S4 will deal with localization of rings and modules, needed in the formulation
and proof of our main theorem, the {\em Telescope Theorem}, Theorem~\ref{thm:telescope}.

We start \S\ref{section:telescopeproof} with the proof of Theorem~\ref{thm:finitelygen}. The rest of \S\ref{section:telescopeproof}
will be devoted to  a  formulation of  Theorem ~\ref{thm:telescope}
together with a number of
results about the telescope of a finitely generated residually nilpotent
metabelian group which will be needed in this paper.

In \S\ref{section:conseq} we will discuss a number of consequences of our Telescope Theorem, namely
Theorems~\ref{thm:finpres}, \ref{thm:parapolycylic} and \ref{thm:parafree}.

In addition to our proof of Theorem~\ref{thm:finitelygenerated},  \S\ref{section:completion} we will briefly discuss some possible
implications of the algorithmic nature of the computation of the Poincar\'e series
involved.

In \S9 we give a number of examples of groups which are completely determined
by their lower central sequences and solve the Isomorphism Problem for
a few finitely generated metabelian groups. The use of the ideal theory needed to
distinguish  some of the groups constructed and the class field theory that comes into
play will also be briefly discussed there. This aspect of our work will be dealt with in
some detail in the third of this series of papers devoted to
 finitely generated metabelian groups.

\section{Preparations for the  telescope theorem}\label{section:telescope}
We term our main result of this paper the {\em Telescope Theorem.}

The {\em telescope of a metabelian group} is a type of group localization, or algebraic closure, whose purpose, from our view, is to turn para-equivalences into equivalences.  The original construction arose from a radically different context in knot theory through work of J. P. Levine~\cite{Levine1,Levine2,Levine3}.  We encourage the reader to investigate Levine's beautiful and powerful constructions independently of this work, and we acknowledge our debt.

When restricting to metabelian groups, Levine's construction has a particularly useful formulation which we investigate here as our definition of the group telescope.  In particular, we define the group  telescope using the classical constructions of ring and module localization, that is, the result of adjoining inverses to  elements in a commutative ring or module.  This classical localization, in turn, extends the construction of a field from an integral domain making use of ``fractions''.

It was the first author who originally proposed employing module localization to examine the Isomorphism Problem for metabelian groups at an NSF funded conference at City College of New York, in March 2011, entitled {\em Finitely Presented Solvable Groups.}  Algebraic geometry motivated this approach. Thus, ideas arising from algebraic geometry and knot theory unexpectedly merge in our exploration of metabelian groups.

\subsection{Localization of rings and modules}
The localizations referred to above
are respectively
\begin{itemize}
\item  the construction of the
ring of fractions $S^{-1}R$  of a unitary commutative ring
$R$  with respect to a multiplicatively closed subset  $S$ of
 $R$ containing the identity element $1$ of $R$

 and

\item a related construction, the module of
fractions $S^{-1}M$ of an
$R$-module $M$ with respect to such a multiplicatively closed
set $S$.
\end{itemize}
We will sometimes refer to these constructions respectively as
rings of fractions or modules of fractions or simply as localizations.
 Our discussion
will follow closely that of Atiyah and MacDonald \cite{AtiyahMacDonald}. We
leave most details and proofs to the reader and refer
simply to the cited reference.

We begin first with the construction of the ring $S^{-1}R$
 which is defined to be the set of
the equivalence classes of elements $(a,s) \in R\times S$ subject to
the equivalence relation
\[
(a,s) \sim (b,t) \, \rm{if\,  there\,  exists}\,  u\in S\,  \rm{such\,  that}
\, (at-bs)u=0.
\]
We denote the equivalence class of $(a,s)\in S^{-1}R$ by
 $\frac{a}{s}$ or, at times, by $a/s$.
$S^{-1}R$ can be turned into a unitary commutative ring in the obvious way:
\[
\frac{a}{s}+\frac{b}{t}=\frac{at+bs}{st}, \qquad \qquad
 \frac{a}{s}\frac{b}{t}=\frac{ab}{st}.
\]
The element $\frac{s}{1}$ is invertible in $S^{-1}R$ with inverse $\frac{1}{s}$.
Notice that $S^{-1}R$ is again
a commutative unitary  ring.

There is an analogous construction  to $S^{-1}R$ where $R$ is replaced
by an $R$-module $M$ and the equivalence relation defined on $R\times S$
is replaced by  an equivalence relation $\sim$ on
$M\times S$ as follows:
\[
(a,s)\sim (b,t) \,\, {\rm if\,\,  there\,\,  exists\, }\,  u\in S\,\, {\rm{such\,
\, that}\,\, (at-bs)u=0}.
\]
The set of equivalence classes of $M\times S$ is denoted by $S^{-1}M$
and we denote the equivalence class of $(a,s)$ by $\frac{a}{s}$ or
$a/s$. $S^{-1}M$
is then turned into an $S^{-1}R$ module in the obvious way by defining
$a/s\cdot r/t=ar/st$. If we now fix $s\in S$ and
consider the $R$-module $M_s=\{\frac{a}{s}\mid a\in M\}$, then the mapping
$\mu_s:a \mapsto \frac{a}{s}$ is monic and maps $M$ isomorphically
onto $M_s \leq S^{-1}M$ provided that $as\neq 0$
for every $a\in M, a\neq 0$. If this condition is satisfied we say that
$S$ does not contain any zero divisors of $M$. In particular if
 $S$ does not contain any zero
divisors of $M$,
then the mapping $a \mapsto a/1$ is monic and we can identify
each element $a \in M$ with $a/1 \in S^{-1}M$.
 If $\alpha:M \longrightarrow N$ is a homomorphism of the $R$-module $M$
into the $R$-module $N$,  then it gives rise to an $S^{-1}R$-module
homomorphism $S^{-1}\alpha: S^{-1}M \longrightarrow S^{-1}N$ defined
by $S^{-1}\alpha: \frac{a}{s} \mapsto \frac{\alpha a}{s}$. Then it follows
that $S^{-1}(\beta\alpha)=(S^{-1}\beta)(S^{-1}\alpha)$.

We will also
denote $S^{-1}R$ by $R_S$ and $S^{-1}M$ by $M_S$ and if $\phi$ is a
homomorphism of the $R$-module $M$ into the $R$-module $N$, we will
denote $S^{-1}\phi$ by $\phi_S$.

\begin{lemma}\label{lemma:flat}
Suppose $M$, $N$ and $P$ are $R$-modules, $R$ a commutative unitary ring. Then the following hold.
\begin{enumerate}
\item If the sequence
\[
0 \to M\xrightarrow{\alpha} N \xrightarrow{\beta}P \to 0
\]
is exact, so too is
\[
0 \to M_S\xrightarrow{\alpha_S}N_S \xrightarrow{\beta_S} P_S \to 0.
\]

\item If $F$ is a free $R$-module with basis $X$, then $F_S$ is a
free $R_S$ module with basis $\{x/1\mid x \in X\}$.
\end{enumerate}
\end{lemma}

The proof of Lemma 4.1 is well known, straightforward, and  omitted.

We now consider the special case where $R=\Z[H]$ is the integral group
ring of the abelian group $H$ and $S=1+I$, where $I$ is the augmentation
ideal of $R$, i.e., the ideal consisting of those elements of $R$
with coefficient sum 0. The proof of the following lemma is
also  straightforward and will be omitted.

\begin{lemma}\label{lemma:locfacts}
Let $M$ and $N$ be $R=\Z[H]$-modules.
Then the following hold.
\begin{enumerate}
\item  $(MI)_S = M_SI$.

\item $M_SI^k = (MI^k)_S$.

\item If $M$ is a submodule of the $R$-module $N$, then
$\left(\frac{N}{M}\right)_S\cong \frac{N_S}{M_S}.$
\end{enumerate}
\end{lemma}

We come next to an important lemma which is needed to verify
one of the properties of our Telescope Theorem.

\begin{lemma}\label{lemma:loc}
Let $H$ be an abelian group and let $M$ and $N$ be $R=\Z[H]$-modules, where $H$ is an abelian group.
Furthermore, let $\sigma$ be a homomorphism of $M$ to $N$. If
$\sigma$ induces a homomorphism from  $M$ onto $N/NI$ and if $N$
is finitely generated, then $\sigma_S$ maps $M_S$ onto $N_S$.
\end{lemma}

\begin{proof}
Let $b_1, \cdots b_k$ be a finite set of generators of $N$ and and
let $F$ be the free $R=\Z[H]$-module with basis $x_1,\dots,x_k$. Then
the homomorphism $\nu$ from $F$ to $N$, defined by sending $x_i$ to $b_i$
for each $i$, is onto $N$. Now $\sigma$ induces a homomorphism from $M$
onto $N/NI$. Consequently, for each $i$ there exists an element $a_i \in M$
such that $\sigma(a_i)=b_i+\Sigma_{j=1}^kb_jr_{ij}$ where the $r_{ij} \in I$.

Define a homomorphism $\lambda$ of $F$ to $F$ in the usual way,
by a $k\times k$ matrix $\Lambda=(\lambda(i,j))$  where
$\lambda(i,j)= \delta_{ij}+r_{ij}$. Finally, define a homomorphism $\rho$  from
$F$ to $M$ by mapping $x_i$ to $a_i$.
It then follows that the following diagram is commutative.
\[
\begin{diagram}\dgARROWLENGTH=1em
\node{F} \arrow{e,t}{\lambda}\arrow{s,r}{\rho} \node{F}\arrow{s,r}{\nu}\\
\node{M} \arrow{e,t}{\sigma}\node{N}
\end{diagram}
\]
Consequently on localizing each of the terms in  the
above diagram, we get another commutative diagram
\[
\begin{diagram}\dgARROWLENGTH=1em
\node{F_S} \arrow{e,t}{\lambda_S}\arrow{s,r}{\rho_S} \node{F_S}\arrow{s,r}{\nu_S}\\
\node{M_S} \arrow{e,t}{\sigma_S}\node{N_S}
\end{diagram}
\]

The determinant of the matrix $\lambda_S$ is invertible since it belongs
to $S$. Consequently $\lambda_S$ is an isomorphism. It follows that
$ \nu_S\lambda_S$ is onto $N_S$, and so the commutativity of the above
diagram implies  that $\sigma_S$ is onto, as claimed.
\end{proof}

\section{The formulation and proof of the telescope theorem}\label{section:telescopeproof}
We give an alternative to our original proof of the Telescope Theorem.  We thank C.F. Miller III for the key Lemma~\ref{lemma:key} used in this newer proof.

Before proceeding to the formulation and proof of the
telescope theorem, we prove
the following theorem.

\begin{theorem}\label{thm:finitelygen}
Let $G$ and $H$ be residually nilpotent metabelian groups.
If $H$ is finitely generated and if there is a
homomorphism $\phi$ of $G$ into $H$ which induces an
isomorphism between their lower central sequences
then $G$ is also finitely generated.
\end{theorem}

\begin{proof}
Let $A$ be the derived group of $G$ and $B$ the derived group of $H$.
Then $\phi$ maps $A$ into $B$ and induces a homomorphism of $G$ onto
$H/B$. So we can find a finite subset $X$ of $G$ whose image $Y$ under $\phi$
generates $H$ modulo $B$. Since $A$ is invariant under conjugation by the
elements of $X$, it follows that $\phi(A)$ is invariant under conjugation
by the elements of $Y$. Since $\phi$ maps $A$ into $B$ it also follows that
$\phi(A)$ is invariant under conjugation by the elements of $H$, i.e.,
that $\phi(A)$ is normal in $H$, and therefore an $H/B$-module.  Since $\Z[H/B]$ is Noetherian, $\phi(A)$ is the normal closure of finitely many elements in $H$. By Theorem~\ref{thm:Magnus}, $\phi$ is a monomorphism.
Therefore  $A$ is the normal closure
in $G$ of finitely many elements. Since $G/A$ is finitely generated, so too is
$G$.
\end{proof}

\subsection{Preliminaries leading to the proof of the Telescope Theorem}
 Let $G$ be an $\M$-group with
 derived group $A$, $S=1+I$, where $I$ is the augmentation ideal of the integral
 group ring $R$ of $H=G/[G,G]$.  $A_S$ is a $G$-module and hence, as before,  $G$
 acts on the localization $A_S$ of $A$. We can then
 form the semi-direct product $P= G \ltimes A_S$ of $A_S$ by $G$.
So, denoting the elements
of $P$ by pairs $(g,a/s)$ (and noticing  the peculiarities of
this notation), $g$ is
 here an element in a multiplicatively
written group $G$ while $a/s$ is an element in a $G$-module now endowed with an
 additive notation.
 Let $K=\{(a^{-1},a/1)  \mid  a\in A \}$.  Then $K$ is a normal subgroup
of $P$. We now define what we term the {\it Telescope} of $G$:

\begin{definition}
 The {\em Telescope  of $G$} is the factor group $G_S=P/K$.
\end{definition}

We need some further preparation before we can formulate in \S5.3,
our main theorem, the Telescope Theorem, Theorem~\ref{thm:telescope}.

\subsection{Some important subgroups of $G_S$}
We will need to identify
some of the subgroups of $G_S$. First we have

\begin{lemma}\label{lemma:GS subgroups}
Let $G$ bs an $\M$-group with derived group $A$
and factor derived group $H=G/A$. Let $R$ be the integral group ring of $H$ and let
$S=1+I$, where $I$ is the augmentation ideal of $R$. Let $G_S$ be the telescope of $G$.
Then
\begin{enumerate}
\item the mapping $\phi_1:G \longrightarrow G_S$ defined by
\[
\phi_1: g \mapsto (g,0/1)\, (g \in G)
\]
is a monomorphism;
\item the mapping $\alpha:A_S \longrightarrow G_S$ defined by
\[
\alpha:a/s \mapsto (1,a/s)\, (a \in A, s\in S)
\]
is  a monomorphism.
\item $G/A\cong G_S/A_S$.
\end{enumerate}
\end{lemma}

\begin{proof}
We begin by noting that if $a \in A$, $s \in S$, then $s=1-\alpha $ with $\alpha \in I$. So
 $as =a-a\alpha$. We note that if $a \neq 0$, then $as \neq 0$ for otherwise
\[
a=a\alpha=a{\alpha}^2=\cdots=a{\alpha}^n=\cdots .
\]
Consequently $a\in \cap_{n=1}^{\infty}A (I^n)$. However $A(I^{n})=\gamma_{n+2}(G)$ which implies
that $a=0$ since $G$ is residually nilpotent. Thus no element of $S$ is a
zero-divisor of $A$. It follows, in particular, that $A_s \cong A$
for every $s \in S$ and in particular that $A_1 \cong A$. So we can,
as needed subsequently, identify $a$ with $a/1$ and thence we identify
$A$ with $A_1$.

The proofs of (1) and (2) of Lemma~\ref{lemma:GS subgroups} are straightforward. We need only
 note that if $(g,0/1)\in K$
then $(g,0/1)=(a^{-1},a/1)$ which implies, interpreting the first
$a$ as an
 element in $G$, that
$a=1$ and hence that $g=1$.

To prove (3) it is  worth-while to first clarify some of the identifications that one
might take for granted working with $G_S$. We
 identify $(g,0/1)$ with $g \in G$ and $a/s$ with
 $(1,a/s)$. Then this identifies $G$ with a subgroup of $G_S$ and $A_S$ with an abelian normal subgroup of $G_S$.
Using these identifications, we see that $G_S=GA_S$, $G\cap A_S=A_1=A=A/1$, and $G_S/A_S \cong G/A$.
\end{proof}

Now recall that for
each $s \in S$, \,  $A_s=\{a/s\mid a\in A\}$.
$A_s$ in $A_S$ is invariant under conjugation by elements of $G$. Viewing, once again, $G$ and $A_s$ as subgroups of $G_S$, we define $G_s$ by the equation $G_s = GA_s$. Thus, the telescope,  $G_S$, is an ascending
union $G_S=\bigcup_{s\in S}G_s$ of the subgroups $G_s$.

In order to prepare for the proof of the  telescope theorem we will need a number of lemmas.

\subsection{A key lemma}
Let $G$ be a metabelian group.

\begin{lemma}[Miller III]\label{lemma:key}
Let $y_1,\dots,y_n$ be elements of $G$ and let $\theta$ be the mapping from $G$ into $G$ defined by
\[
x \mapsto x[x,y_1]\cdots [x,y_n]\, (x\in G).
\]
Then $\theta$ is an endomorphism  of $G$ which
induces the identity on $G/[G,G]$.
\end{lemma}

\begin{proof}
Suppose that $x$,  $y$  and $z$ are elements of $G$. Then
$[xz,y]=[x,y]^z [z,y]$. Hence, noting that commutators commute in $G$, we find that
\begin{align*}
\theta(xz)
& =xz[xz,y_1][xz,y_2]\cdots [xz,y_n]\\
&=xz[x,y_1]^z[z,y_1]\cdots [x,y_n]^z[z,y_n]\\
&=xz([x,y_1][x,y_2]\cdots [x,y_n])^z [z,y_1][z,y_2]\cdots [z,y_n]\\
&=x[x,y_1][x,y_2]\cdots [x,y_n]z[z,y_1][z,y_2]\cdots [z,y_n]\\
&=\theta(x)\theta(z)
\end{align*}
\end{proof}

Notice that the definition of $\theta$ does not depend on the
 ordering of the set $\{y_1,\dots,y_n\}$.

We now have the following important consequences of Lemma~\ref{lemma:key}.

\begin{lemma}\label{lemma:keyconseq}
Let $G$ be an $\M$-group and let $A$ be the derived
group of $G$. Furthermore,
let $I$ be the augmentation ideal of the integral group ring $R=\Z[H]$ of $H=G/A$. Let $S=1+I$ and let  $t\in S$. Then
$t=1+\alpha$, where $\alpha \in I$ and so having chosen a set $Y$ of representatives of the
cosets of $A$ in $G$, $\alpha$ can, disregarding order, be written uniquely as
$\alpha=(-1+y_1A)+\cdots+(-1+y_nA)\, (y_j\in Y)$. Then the mapping $\theta_t$ defined by
\[
x \mapsto x[x,y_1]\dots [x,y_n]\, (x\in G_t)
\]
is an endomorphism of
$G_t$ which is monic,  maps $A_t$ isomorphically onto $A$, and
$G_t$ isomorphically onto $G$. Hence
\[
G_t \cong G,\quad A \gamma_2(G_t) = A_t\cong A \cong \gamma_2(G).
\]
\end{lemma}

\begin{proof}
The proof here is most easily given by using multiplicative notation for $R$-modules. Thus we
use the notation $(a/t)^r$ in place of $(a/t)r$. Moreover, if $r=c_1g_1+\cdots+c_ng_n$, then
$(a/t)^r=(a^{g_1})^{c_1}\cdots(a^{g_n})^{c_n}$. Notice that
\[
\theta_t(a/t)=a/t\cdot [a/t,y_1]\cdots [a/t,y_n]=(a/t)^{1+(-1+y_1)+\cdots+(-1+y_n)}=(a/t)^t=a.
\]
Now every element $x\in G_t$ can be written in the form $x=g\cdot (a/t)$, where $g \in G$ and
$a\in A$. Then
\[
\theta_t(g\cdot a/t)=\theta_t(g)\theta_t(a/t)= \theta_t(g)\cdot (at/t)=\theta_t(g)\cdot a.
\]
Since $\theta_t$ maps $G$ into $G$
it follows that $\theta_t$ maps $G_t$ into $G$.
Moreover modulo $A_t$, $\theta_t$ is the identity and is monic on $A_t$.
So $\theta_t$ is an isomorphism as
claimed. The rest of the lemma follows easily.
\end{proof}

Notice that the definition of the mapping $\theta_t$
 depends on the choice of the
set of representatives of $A$ in $G$.

We now, adopting the notation used above, are
in a position to prove the following lemma.

\begin{lemma}\label{Lemma:structure thm}
Let

\[
S=\{s_1,s_2,\dots,s_n,\dots\},
\]
and
\[
T=\{t_1, t_2,\dots, t_n, \dots\} \ \mbox{where for each $n$}, \  t_n=s_1\dots s_n.
\]
If we denote $G = G_0$ and $G_i = G_{t_i}$, then the following hold:
\begin{enumerate}
\item $G=G_0 \leq G_1\leq G_2\leq \dots \leq G_S$ is an ascending
sequence of subgroups of $G_S$;
\item $G_i \cong G$ for all $i$;
\item $G_S =\bigcup_{i=1}^{\infty} G_i$;
\item $\gamma_2(G_S)=A_S$;
\item $A/AI \cong  A_S/A_SI$;
\item $A/AI^n\cong A_S/A_SI^n$;
\item $\g_{n+2}(G_S)= A_SI^n=(AI^n)_S$ , $\gamma_{n+2}(G)=AI^n$,  for every $n$;
\item  $ \gamma_2(G)/\gamma_{n+2}(G)\cong \gamma_2(G_S)/\gamma_{n+2}(G_S)$;
\item $G_S$ is residually nilpotent.
\end{enumerate}
\end{lemma}

\begin{proof} We prove the above statements sequentially.
\begin{enumerate}
\item This follows immediately from the very definitions of the $G_i$ since
\[
A_{i+1}=A_{t_{i+1}}=A_{t_i \cdot s_{i+1}}\geq A_{t_i} = A_i.
\]

\item This is a special case of Lemma~\ref{lemma:keyconseq}.\\

\item This follows from Lemma~\ref{lemma:GS subgroups}, as discussed in the paragraph subsequent to that Lemma.\\

\item Recall that the isomorphism $\theta_t \colon G_t \to G$ of Lemma 5.5 maps  $A_t$ to the derived subgroup $A$ of $G$.  It follows that the inverse of $\theta_t$ maps $A$
to the derived subgroup of $G_t$ and so  $\gamma_2(G_t) = A_t$. Consequently
\[
\gamma_2(G_S)=\gamma_2(\cup_{t_i}G_{t_i})=\cup_{t_i}\gamma_2(G_{t_i})=
\cup(A_{t_i})=A_S.
\]

\item We have the following sequence of isomorphisms.
\[
\begin{split}
A_S/A_SI\cong A_S\otimes_{Z[H]}\Z \cong (A \otimes_{\Z[H]}\Z[H]_S)&\otimes_{\Z[H]} \Z\\
\cong  A \otimes_{\Z[H]} (\Z[H]_S  \otimes_{\Z[H]} &\Z)\cong A\otimes_{\Z[H]}\Z \cong A/AI.
\end{split}
\]

\item The case $k=1$ is taken care of by (5). Suppose inductively that
$A_S/A_SI^k\cong A/AI^k$ for $k = n$.  We have a commutative diagram
where the second row is exact by the flatness of localization, Lemma~\ref{lemma:flat}.
\[
\begin{diagram}\dgARROWLENGTH=1em
\node{0} \arrow{e} \node{\frac{AI^{n}}{AI^{n+1}}} \arrow{e} \arrow{s,r}{\cong} \node{\frac{A}{AI^{n+1}}} \arrow{e}\arrow{s} \node{\frac{A}{AI^{n}}} \arrow{s,r}{\cong}\arrow{e} \node{0}\\
\node{0} \arrow{e} \node{\left(\frac{AI^{n}}{AI^{n+1}}\right)_S} \arrow{e} \node{\left(\frac{A}{AI^{n+1}}\right)_S} \arrow{e} \node{\left(\frac{A}{AI^{n}}\right)_S} \arrow{e} \node{0}
\end{diagram}.
\]
The right hand vertical homomorphism is an isomorphism by the inductive hypothesis and Lemma~\ref{lemma:locfacts}. The following shows the left hand vertical homomorphism is also an isomorphism:
{\small
\[
\frac{AI^{k}}{AI^{k+1}} \cong AI^{k} \otimesover{\Z [H]} \Z \cong
(AI^{k})_S\otimesover{\Z [H]} \Z \cong (A_S)I^{k} \otimesover{\Z
[H]} \Z \cong \frac{(A_S)I^{k}}{(A_S)I^{k+1}} \cong
\left(\frac{AI^k}{AI^{k+1}}\right)_S.
\]
}
By the $5$-Lemma, the result follows.\\

\item By statement (4), already proven, $\gamma_2(G_S) = A_S.$  By induction
\[
\gamma_{n+2}(G_S) = [G_S, \gamma_{n+1}(G_S)] = [G_S, A_SI^{n-1}] = [G, A_SI^{n-1}] = A_SI^n = (AI^n)_S.
\]
The middle equality follows since the derived subgroups of $G$ and $G_S$ act trivially on $A_S$, and by statement (4) which implies that $G/A \cong G_S/A_S$.\\
\item We have $A=\gamma_2(G)$, $AI^n=\gamma_{n+2}(G)$, $A_S=\gamma_2(A_S)$,
$A_SI^n=\gamma_{n+2}(G_S)$ and so by (6)
$$\gamma_2(G)/\gamma_{n+2}(G)=A/AI^n\cong A_S/A_SI^n=\gamma_2(G_S)/\gamma_{n+2}(G_S)
.$$
\item Since $G$ is residually nilpotent the submodules $AI^n = \gamma_{n+2}(G)$
have trivial intersection and so it follows from (7) that $G_S$ is residually nilpotent.
\end{enumerate}
\end{proof}

Before we state and prove our Telescope Theorem, we establish some notation.

\begin{notation}
For an $\M$-group $G$, we denote the telescope of $G$ by $\tau(G)$.
\end{notation}

We are now in a position to formulate and prove our Telescope Theorem,
much of which has already been proved in Lemma~\ref{Lemma:structure thm}.

\begin{theorem}[The Telescope Theorem]\label{thm:telescope}
Let $G$ and $H$ be $\M$-groups. Then the
following hold:
\begin{enumerate}
\item $\tau(G)$ is an ascending union of subgroups isomorphic to $G$.
\item The telescope $\tau(G)$ is para-$G$.
\item The group $H$ is para-$G$ if and only if $\tau(G) \cong \tau(H)$.
\item If $H$ is para-$G$, then $G$ is para-$H$.
\item If $H$ is para-$G$, then $H$ isomorphic to a subgroup of $G$ and
$G$ is isomorphic to a subgroup of $H$.
\end{enumerate}
\end{theorem}

\begin{proof} We prove these statements in sequence.
\begin{enumerate}
\item We already know from~Lemma~\ref{Lemma:structure thm}(1) that $\tau(G)$ is an ascending union of the $G_i$,
each of which is isomorphic to $G$.\\

\item  Let $\phi$ be the inclusion of $G$ into $\tau(G) = G_S$ and let $\phi_n$ be the induced homomorphism
of $G/\gamma_n(G)$ into $\tau(G)/\gamma_n(\tau(G))$.
Observe  that $\gamma_2(\tau(G))= A_S$ by Lemma~\ref{Lemma:structure thm}, and recall that
$\gamma_2(G)=A$.
Consider the following  sequence of equalities and isomorphisms
$$G/\gamma_2(G)= G/A\cong GA_S/A_S=GA_S/\gamma_2(G_S)=\tau(G)/\gamma_2(\tau(G)).$$
The composition of these maps is $\phi_2$. It follows that if $X$ is a finite set of generators of $G$,
then $\{x\gamma_2(G_S)\mid x \in X\}$ generates $G_S/\gamma_2(G_S)$. Now in a nilpotent group,
any set of elements which generates it modulo its derived group generates it modulo every term of its
lower central series. Hence
$$\phi_{n+2}: G/\gamma_{n+2}(G) \longrightarrow G_S/\gamma_{n+2}(G_S)$$
is onto.
Notice that  since $G$ is finitely generated
so too are all of its quotients and subgroups. Moreover if $H$ is any finitely generated nilpotent group and
if $H/L \cong H$ then $L=1$ by Magnus' Theorem ~\ref{thm:Magnus}. It follows that any homomorphism of a finitely generated
nilpotent group onto an isomorphic nilpotent group is itself an isomorphism. Now $\phi_2$ induces a homomorphism of $G/\gamma_2(G)$
onto the group $G_S/\gamma_2(G_S)$ which we have already seen is isomorphic to $G/\gamma_2(G)$.
So $\phi_2$ is an isomorphism.
$\phi_n$ induces a homomorphism $\theta_n$
of $A/AI^n=\gamma_2(G)/\gamma_{n+2}(G)$ to
 $A_S/A_SI^n=\gamma_2(G_S)/\gamma_{n+2}(G_S)$, which by Lemma 5.6 (6) is
isomorphic to $A/AI^n$. Each
of these groups is finitely generated since subgroups and quotient groups of finitely
generated nilpotent groups are finitely generated and hence $\theta_n$ is  an isomorphism
and so
$$\phi_n: G/\gamma_n(G)\longrightarrow  \tau(G)/\gamma_n(\tau(G))$$
is an isomorphism for all $n$ which proves (2).
\item Suppose $H$ is para-$G$.  Then there is a homomorphism $\phi$ from
$G$ into $H$ which induces isomorphisms of the corresponding
terms of the lower central sequences of $G$ and $H$. In particular,
$G/[G,G]$ and $H/[H,H]$ are isomorphic finitely generated abelian groups, which we identify
and denote by $Q$.

So if we put $\g_2(G) = A$ and $\g_2(H)=B$ then it follows that $\phi$ induces a
homomorphism of the $Q$-module $A$ into the $Q$-module $B$.
Moreover $\g_3(G)=AI$ and $\g_3(H)=BI$. So
$\phi$ induces a homomorphism of $A/AI$ onto $B/BI$. By
Lemma~\ref{lemma:loc},  $\phi$ induces an
epimorphism of $A_S$ onto $B_S$. It follows that $\phi$
induces a  homomorphism $\phi_S$ of $G_S$ onto
$H_S$. But $G_S$ and $H_S$ have the same lower central
sequences and they are residually nilpotent by Lemma~\ref{Lemma:structure thm}(8).
Therefore, we have proved
that $\tau(G)\cong \tau(H)$.

 Choose any isomorphism $\phi \colon \tau(G) \to \tau(H)$.
By Lemma~\ref{Lemma:structure thm}(1) and (3), $\tau(H)$ is a
union of subgroups $H_k$ with  $H_k \cong H$.
Since $G$ is finitely generated, $\phi$ sends  $G$ into one of these copies of $H$, say $H_k$.
We will show $H_k$ is para-$G$.  Since $H \cong H_k$, this will suffice to prove  that $H$ is para-$G$.

To show that $G \to H_k$ induces an isomorphism on lower central series quotients,
first note that the following commutative diagram implies
that $G/\gamma_n(G) \to H_k/\gamma_n(H_k)$ is one-to-one for all $n$.
\[
\begin{diagram}\dgARROWLENGTH=1em
\node{G/\gamma_n(G)} \arrow{e,t}{\phi}\arrow{s,r}{\cong} \node{H_k/\gamma_n(H_k)}\arrow{s}\\
\node{\tau(G)/\gamma_n(\tau(G))} \arrow{e,t}{\cong} \node{\tau(H)/\gamma_n(\tau(H))}
\end{diagram}
\]

Let $B_k$ be the copy of $B$ in $H_k$. The homomorphism $H_k \to \tau(H)$ is given by $H_k = HB_k \to HB_S = \tau(H)$.  Thus,
\[
\qquad \qquad H_k/\gamma_2(H_k) = HB_k/B_k \cong H/H\cap B_k = H/H \cap B_S \cong HB_S/B_S = \tau(H)/\gamma_2(\tau(H)).
\]
Hence, $G/\gamma_2(G) \to H_k/\gamma_2(H_k)$ is an isomorphism, and in particular, onto.  This implies that $G/\gamma_n(G) \to H_k/\gamma_n(H_k)$ is onto for all $n$ since, as previously noted, any set of elements in a nilpotent group which generate it modulo its derived group generates the group itself.
It follows alomg the same lines in our precious discussions that
$\phi \colon G/\gamma_n(G) \cong H_k/\gamma_n(H_k)$ for all $n$. Consequently, $H_k$, and
hence $H$, is para-$G$.\\

\item If $H$ is para-$G$ then $\tau(G) \cong \tau(H)$.  So $\tau(H) \cong \tau(G)$ which implies $G$ is para-$H$.\\

\item This follows immediately from statement~(4) above and Magnus' Theorem~\ref{thm:Magnus}.
\end{enumerate}
\end{proof}

\section{Some consequences of the telescope theorem}\label{section:conseq}

\begin{theorem}\label{thm:finpres}
Suppose that $G$ and $H$ are $\M$-groups
and that $H$ is para-$G$. Then $H$ is finitely presented if and only if $G$ is finitely
presented.
\end{theorem}

We will need a theorem of Bieri and Strebel~\cite{BieriStrebel}. We recall the details.  Let $Q$ be
a finitely generated abelian group and let $v \in Hom(Q,\R)$. Define the submonoid
$Q_v=\{q \in Q\mid v(q) \geq 0\}$.
Now let $A$ be a finitely generated $\Z[Q]$-module. Then for every
$v \in Hom(Q,\R)$, $A$ can be viewed as
a $\Z[Q_v]$-module. $A$ is termed {\it tame} if for every
$v\in Hom(Q,\R)$ either $A$ is finitely generated as a $\Z[Q_v]$-module or else it is finitely
generated as a $\Z[Q_{-v}]$-module. The relevance of this is the following
theorem of Bieri and Strebel.

\begin{bs}
Suppose $Q$ is a finitely generated abelian group.
\begin{enumerate}
\item
If $A$ is a finitely generated tame
$\Z[Q]$-module, then every $Q$-submodule of $A$ is also tame.
\item
Suppose that $G$ is an
extension of an abelian normal subgroup $A$ by
$Q$. Then $G$ is finitely presented
if and only if $A$ is a tame $\Z[Q]$-module.
\end{enumerate}
\end{bs}

We are now in a position to prove Theorem~\ref{thm:finpres}.

\begin{proof}[Proof of Theorem~\ref{thm:finpres}]
Put $B=[H,H]$, $Q=H/[H,H]$. Since $H$ is para-$G$, there exists a homomorphism
$\phi$ from $G$ into $H$ which induces isomorphisms between the corresponding quotients of
their lower central series. So $\phi$ induces an isomorphism between $G/[G,G]$ and $Q$.
Since $H$ is para-$G$,  $\phi$ induces a monomorphism of $A=[G,G]$. It follows that $\phi(A)$
is a normal subgroup of $H$, i.e., a $Q$-submodule of the $Q$-module $B$.

Now suppose that $H$ is finitely presented. Then $B$ is tame and therefore so too is every
submodule of $B$, in particular $\phi(A)$. It follows that the module $A$ is also tame
which by the Bieri-Strebel Theorem implies that $G$ is finitely presented.

Conversely since para-equivalence is an equivalence relation, $G$ is para-$H$. So again
as already noted, if $G$ is finitely presented, so too is $H$.
\end{proof}

We now record another, simple consequence of the Telescope Theorem, for
polycyclic groups.

\begin{theorem}\label{thm:parapolycylic}
Suppose that $H$ is a finitely generated residually nilpotent
 metabelian group.  If $G$ is polycyclic and $H$ is para-$G$, then
 $H$ is isomorphic to a subgroup
of finite index in $G$ and $G$ is isomorphic to a subgroup of finite index in $H$.
\end{theorem}

\begin{proof} By the Telescope Theorem~\ref{thm:telescope}(5), $H$ is isomorphic to
a subgroup of $G$. It follows that the Hirsch number of $H$ is less than or equal to that of $G$.
But again by the Telescope Theorem~\ref{thm:telescope}(5), $G$ is isomorphic to a subgroup of $H$ and so the Hirsch
number of $G$ is less than or equal to that of $H$. It follows that $G$ and $H$ have the same
Hirsch numbers and hence the image of $H$ in $G$ and that of $G$ in $H$ are of finite index.
This completes the proof of Theorem~\ref{thm:parapolycylic}.
\end{proof}

We record and prove one last consequence of the Telescope Theorem.

\begin{theorem}\label{thm:parafree}
Let $G$ be an $\M$-group.
If $G$ has the same lower central quotients as a finitely generated
free metabelian group $F$, i.e., $G$ is para-free-metabelian,
then $G$ is isomorphic to a subgroup of $F$.
\end{theorem}

We will only sketch the proof of Theorem~\ref{thm:parafree}. Suppose that $G$ is freely generated
modulo $[G,G]$ by the set $g_1,\dots,g_n$ and that $F$ is freely generated by $x_1,\dots,x_n$.
Define a homomorphism $\phi$ from $F$ into $G$ by sending $x_j$ to $g_j$
for each $j$. Then $\phi$ induces isomorphisms between the corresponding quotients
of the lower central series of $F$ and $G$. Since free metabelian groups are residually
nilpotent, it follows that $G$ is para-$F$. Hence, by the Telescope Theorem~\ref{thm:telescope}(4), $F$ is
para-$G$. Consequently $G$ is isomorphic to a subgroup of $F$.
\qed \\

We note that there exist finitely generated para-free-metabelian groups which are not
free~\cite{Baumslag3}. We shall discuss the existence of more para-free metabelian  groups in
a separate paper.

\section{Poincar\'e series for finitely generated, metabelian groups}\label{section:poincare}

Our objective here is to prove Theorem~\ref{poincare}.  We will need the following
\begin{lemma}\label{lem:presentation}
Let $G$ be a finitely generated metabelian group given by a preferred presentation.
There is an algorithm to compute a finite presentation for
$G/\gamma_{n+1}(G)$ for every $n$.
\end{lemma}
In view of the fact that preferred presentations are not finite but involve
infinitely many relations that ensure that the derived group is abelian, some
care is needed to obtain a  finite  presentation of $G/\gamma_{n+1}(G)$.
To this end we first  recall the definition of a preferred presentation:

\begin{definition}
A  {\em preferred presentation} of a finitely generated metabelian group $G$, in the category of metabelian groups, is a presentation
which takes the form

\[
G=\langle g_1,\dots,g_t\ |\ R_1\cup R_2\rangle
\]
where
\begin{enumerate}
\item $R_1$ is a finite set of words of the form
\[
w=\prod_{1\leq i\leq j \leq   t}  [g_i,g_j]^{u_{ij}}\,
\]
and we use the usual notation $[x,y]$ for $x^{-1}y^{-1}xy$, $y^x$ for $x^{-1}yx$ and the $u_{ij}$ are words of the form
$g_1^{m_1}\dots g_t^{m_t}$;
\item $R_2$ is a finite set of words $uw$ where  $u$ has the form $g_1^{m_1}\dots g_t^{m_t}$ and
\[
w=\prod_{1\leq i\leq j\leq t }[g_i,g_j]^{v_{ij}}\, ,
\]
with $v_{ij}$ of the form $g_1^{n_1}\dots g_t^{n_t}$.\\
\end{enumerate}
\end{definition}
\noindent Thus the words in $R_1$,  together with the addition of all commutators  $[x,y]$ where $x$ and $y$ take
the form $[g_i,g_j]^{u_{ij}}$, are a regular  presentation of $A$ while those in $R_2$ read modulo $A$,
yield a finite presentation of $H = G/A$.

In order to obtain a presentation for $G$ {\em in the category of all groups}
we need to add to the relations $R_1$ and $R_2$ the set $R_3$ of all relations of the form
\[
[[w,x],[y,z]]=1,
\]
where $w,x,y,z$ range over
all words in the generators of $G$.
\begin{proof}[Proof of Lemma~\ref{lem:presentation}]
Our objective is to prove that there is an algorithm to obtain a finite
presentation of $G/\gamma_{c+1}(G)$. To this end we start out by finding
a finite presentation of a free nilpotent group of class $c$ on the generators
of $G$ and add the relations $R_1$ and $R_2$ to the group $H$ presented in this
way.  Now let $d_1,d_2,\dots,d_n, \dots$ be a recursive enumeration of
 the relators in $R_3$
and let $D_n$ be the normal closure in $H$ of $d_1,d_2,\dots,d_n$.
Then
$$D_1\leq D_2 \dots \leq D_n \dots $$
is an increasing sequence of normal subgroups of the finitely generated
nilpotent group $H$.  It follows that for some $m$, adding the finite set of relations $D_m$ will give a presentation for $G/\gamma_{c+1}(G)$.  There is an algorithm to determine for each $m$ whether $D_m = D_{m+1}$~\cite[Lemma~2.2]{BCR}.  Let $k$ be the smallest integer such that $D_k = D_{k+1}$.  Thus, adding the finite set of relations in $D_k$ gives the desired finite presentation for $G/\gamma_{c+1}(G)$
as desired, and completes the proof of the lemma.
\end{proof}

We continue our discussion concerning Theorem~\ref{poincare}.

We have proved that if $G$ is any finitely
generated metabelian group given by a preferred presentation,
then we can recursively enumerate fiinte presentations for the lower central series quotients. The basic commutators generate the subgroup $\gamma_c(G)/\gamma_{c+1}(G) \leq G/\gamma_{c+1}(G)$.  Since quotients of the derived group are submodule computable~(see, for instance,~\cite{BCM}), we can algorithmically generate a presentation for $\gamma_c(G)/\gamma_{c+1}(G)$ as an abelian group. Thus we can compute the rank of $\gamma_c(G)/\gamma_{c+1}(G)$.
Consequently we can recursively enumerate the rational Poincar\'e series
of a finitely generated metabelian group.

We are left for the proof of Theorem~\ref{poincare} to compute the denominator and numerator of  $P(G)$ in the case
where $G$ is a free metabelian group of finite rank. This is straightforward and we leave it as an exercise
for the reader.

\section{Completions of polycyclic groups}\label{section:completion}
We prove a sequence of Theorems and Lemmas culminating in the proof of the following

\begin{theorem}\label{thm:finitelygenerated}
Let the $\M$-group $G$ be polycyclic.
Then the pro-nilpotent completion of $G$, $\widehat{G}$, is locally polycyclic, i.e., its finitely generated
subgroups are polycyclic.
\end{theorem}

Our first Theorem may have independent interest.

\begin{theorem}\label{thm:2generators}
A finitely generated metabelian group is polycyclic if and only if
its two-generator subgroups are polycyclic.
\end{theorem}

\begin{proof}
Since subgroups of polycyclic groups are polycyclic, we only need prove one direction of this theorem.

Suppose  that the two-generator subgroups of the finitely
generated  metabelian group $G$ are polycyclic and that $G$ is
generated by $x_1,\dots,x_{\ell}$. Then the commutator subgroup
$A=[G,G]$ of $G$ is the normal closure of finitely many elements,
say $a_1,\dots,a_m$.  Since the subgroup of $G$ generated by $a_1$
and $x_1$ is polycyclic, so is the subgroup $A_1$
of $A$ which is generated by the conjugates of $a_1$ by the powers of $x_1$. Call these generators $a(1,1), \dots, a(1,n_1)$.  The subgroup of $A$ generated by the
conjugates of the finitely many elements $a(1,1), \dots, a(1,n_1)$
by the powers of $x_2$ is finitely generated as well. Iterating this
process we find that the subgroup $B$ of $A$ generated by the
conjugates of the elements $a_1, \dots, a_m$ by the finitely
many elements $x_1,\dots,x_{\ell}$ is finitely generated.

But $B=A$, the derived group of $G$. Consequently $G$ is an extension
of one finitely generated abelian  group by another and is
therefore polycyclic. This completes our proof.
\end{proof}

\begin{lemma}\label{lem:alphabeta}
Let $G$ be a polycyclic metabelian group.   Let $A$ be an abelian
normal subgroup of $G$ with abelian factor group $Q=G/A$. View $A$
as a module over the integral group ring $R$ of $Q$. Then for each
$t=sA\in Q, a \in A$, there exist polynomials
\[
\alpha=c_0+c_1t+\dots +c_{m-1}t^{m-1}-t^m
\]
and
\[
\beta=-t^{-1}+d_0+d_1t+\dots + d_{n-1}t^{n-1}
\]
such that
\[
a\alpha=a\beta=0.
\]
Moreover given two such polynomials $\alpha$ and $\beta$ if
$a\alpha=a\beta=0$, then it follows that the conjugates of $a$ by
the powers of $s$ generate an abelian group which can be generated
by  $m+n$ elements.
\end{lemma}

\begin{proof} We construct $\alpha$ and $\beta$ is two steps.

\begin{enumerate}
\item Consider the subgroup $B_i$ of  the subgroup $gp(a,s)$ of
$G$  generated by
\[
a,a^s,\dots,a^{s^i}.
\]
Then since $gp(a,s)$ is polycyclic, it satisfies the maximal
condition, i.e. every subgroup is finitely generated. So there exists an integer $m$ such that $a^{s^m}\in
B_{m-1}$.  Hence there exists $c_0,\dots,c_{m-1}$ such that
\[
a^{s^m}=a^{c_0}+a^{c_1s}+\dots + a^{c_{m-1}s^{m-1}}.\]

Hence $a\alpha =0$ as claimed.\\

\item Consider the subgroup $C_j$
generated by
\[
a^{s^{-j} },a^{s^{-j}+1},\dots,a^{s^{-1}}.\]

Since $G$ satisfies the maximal condition there exists an integer
$n$ such that
\[
a^{s^{-n}}\in C_{n-1}.
\]
So there exist $d_0,d_1,\dots d_n$ such that
\[
a^{s^{-n}}=a^{d_0s^{-n+1}} + a^{d_1s^{-n+2}} +\cdots + a^{d_{n-1}}.\
\]
Since the action of $t$ on $A$ is by conjugation by $s$, we can
re-express what we have proved by writing $a\beta=0$ as claimed
\end{enumerate}
\end{proof}

Now let $G$ be an $\M$-group.  Our objective  is to prove that the finitely generated
subgroups of $\widehat{G}$ are polycyclic. In view of Theorem
\ref{thm:2generators} it suffices to prove that the two-generator
subgroups of $\widehat{G}$ are polycyclic. The following simple lemma
will facilitate the proof.

\begin{lemma}\label{lem:technical[s,a]}
Let $H$ be a metabelian group generated by the elements $s$ and
$a$. Then $H$ is polycyclic if the subgroup generated by
conjugates of $[s,a]$ by the powers of $s$ is finitely generated
and the subgroup generated by the  conjugates of $[s,a]$ by the
powers of $a$ is finitely generated.
\end{lemma}

\begin{proof}
Notice that $[H,H]$ is the normal closure in $H$ of $[s,a]$. Let
$h_1,\dots,h_m$ be a finite set of generators of the subgroup of
$H$ generated by the conjugates of $[s,a]$ by the powers of $s$.
Notice that $([s,a]^s)^a=([s,a])^a)^s$. So the subgroup  $K$ of
$H$ generated by the conjugates of the elements $h_1,\dots,h_m$ by
the powers of $a$ is again finitely generated.  But $K=[H,H]$. Thus
$H$ is an extension of one finitely generated abelian group by
another finitely generated abelian group and therefore polycyclic.
\end{proof}

To prove Theorem~\ref{thm:finitelygenerated}, we restrict our attention to the case of a two-generator metabelian
group $G$ since the general case follows along the same lines. If $G$ is a metabelian group, then so is $\widehat G$.  To prove that in $\widehat G$ the finitely generated subgroups are polycyclic, it suffices to show that the two-generator subgroups of $\widehat G$ are polycyclic by Theorem~\ref{thm:2generators}.  By Lemma~\ref{lem:technical[s,a]}, it suffices to show that if $s,a\in \widehat{G}$ and $H=gp(s,a)$, then the subgroups of $H$ generated by the conjugates of $b=[s,a]$ by $s$ is finitely generated, and similarly, the subgroup of $H$ generated by conjugates of $b$ by the powers of $a$ is finitely generated.

Toward this end, we have one last Lemma before we prove Theorem~\ref{thm:finitelygenerated}.

\begin{lemma}
Let $s\in \widehat{G}$ and let $b\in [\widehat{G},\widehat{G}]$. Then the subgroup
$B$ of $\widehat{G}$ generated by the conjugates of $b$ by the powers
of $s$ is a finitely generated abelian group.
\end{lemma}
\begin{proof}
Recall that we assume $G$ is generated by two elements, say, $x_1,x_2$. Then
$\g_n(G)/\g_{n+1}(G)$ is generated by the right-normed commutators
of the form
\[
[x_1, x_2, y_1, \dots,y_{n-2}]\g_{n+1}(G),
\]
where the $y_j\in \{x_1,x_2\}$.
Since $G$ is metabelian, so is $\widehat{G}$. So in order to prove that the finitely
generated subgroups of $\widehat{G}$ are polycyclic, it suffices to
prove that the two-generator  subgroups of $\widehat{G}$  are
polycyclic.

Let
\[
s(n)=s_1\dots s_n \g_{n+1}(G), \quad b(n)=b_1\dots b_n\g_{n+1}(G),
\]
where here $s_j\in \g_j(G),\,  b_j \in \g_j(G)$ for each $j$. If
$s_1 \in \g_2(G)$, then $s$ and $b$ commute and there is
nothing to prove.

We consider, then, the case where $s_1 \notin
\g_2(G)$. We need to consider the elements  $b_n$. To this
end, let us denote by $Y_n$ the set of  commutators of the form
$z(y_1,\dots,y_{n-2})=[x_1,x_2,y_1,\dots,y_{n-2}]$ of weight $n>1$.
We adopt the convention that if $n=2$, then $z=z(y_1,y_2,\dots,y_{n-2})=[x_1,x_2]$.

We proved in Lemma~\ref{lem:alphabeta}
that there exist two polynomials $\alpha, \beta$ in
$s_1,s_1^{-1}$, where
\[
\alpha=c_0+c_1s_1+\dots +c_{m-1}{s_1}^{m-1}-s_1^m
\]
and
\[
\beta=-{s_1}^{-1}+d_0 + d_1{s_1}+\dots+d_{n-1}{s_1}^{n-1}
\]
such that
\[[x_2,x_1]{\alpha}=[x_2,x_1]{\beta}=0.
\]
Now each $z(y_1, y_2, \ldots, y_{n-2})$ can be rewritten using exponential notation as
\[
[x_1,x_2]^{(y_1-1)\dots (y_{n-2}-1)}.
\]
So it follows that the action of $\alpha$ on $z(y_1, y_2, \ldots, y_{n-2})$ can be
re-expressed as follows:
\[([x_2,x_1]^{(y_1-1)\dots  (y_{n-2}-1)})^{\alpha}=[x_2,x_1]^{\alpha (y_1-1)\dots (y_{n-2}-1)}.
\]
So
\[
([x_2,x_1]^{(y_1-1)\dots  (y_{n-2}-1)})^{\alpha}=1.
\]
It follows that $b{\alpha}=0$ and similarly that $b{\beta}=0$.
Thus the conjugates of $b$ by the powers of $s$ generate a
finitely generated group. This completes the proof.
\end{proof}

We come now to the proof of Theorem~\ref{thm:finitelygenerated}.

\begin{proof}[Proof of Theorem~\ref{thm:finitelygenerated}]
The same proof used above can be used to prove that the conjugates of $b$ by
the powers of $a$ also generate a finitely generated group. So Lemma
\ref{lem:technical[s,a]} applies as noted above. Thus we have proved that if
$G$ is polycyclic, the two-generator subgroups of $\widehat{G}$ are
polycyclic and hence the finitely generated subgroups of $\widehat{G}$
are also polycyclic by Theorem~\ref{thm:2generators}, as claimed.
\end{proof}

\section{Examples}\label{section:examples}
We give here a number of examples of residually nilpotent metabelian groups with the same
lower central sequences and with a variety of different properties.

\subsection{Wreath products}
We recall that a group $W$ is the {\em (restricted) wreath product} of its subgroups $A$ and $T$,
denoted by $A \wr T$, if $W$ is  generated by $A$ and $T$ and
\begin{enumerate}
\item the conjugates $A^t$ of $A$ by the distinct  elements $t\in T$ generate a (restricted) direct product $B$, and
\item  $A \cap B =1$.
\end{enumerate}
So $W=B\rtimes T$. From now on we will refer to these products as direct products
 and wreath products.
By a  direct product we mean the group of elements of the cartesian product
$\prod_{t \in T} A^t$ where all but finitely many coordinates are the trivial element.
We will prove that such wreath products have a number of interesting properties.

We will restrict our attention here to a number of special cases.  It is likely that
most of our results hold more generally, but we will
not concern ourselves with greater generality here.

One can extend Theorem~\ref{thm:wreath}(1), below, to groups where $A$ is any finite abelian group, as stated in \S\ref{section:main}.  We prove the less general result here.

\begin{theorem}\label{thm:wreath}
Let $W=A \wr T$ be the wreath product of its subgroups $A$ and $T$.
Then the following
hold.
\begin{enumerate}
\item  If $A$ is of prime order and $T$ is infinite cyclic, then any finitely generated
residually nilpotent metabelian group $H$ with the same lower central sequences as
$W$ is isomorphic to $W$, i.e., the para-equivalence class of $W$ consists of a single
isomorphism class.
\item If $A$ is of prime order and $T$ is infinite cyclic, then
the Isomorphism Problem is solvable for $W$.
\item If $A$ and $T$ are infinite cyclic, then the Isomorphism Problem is solvable
for $W$.
\item If $A$ and $T$ are infinite cyclic then the para-equivalence class of
$W$ contains at least two non-isomorphic groups.
\end{enumerate}
\end{theorem}

These groups are residually nilpotent, as claimed in \S\ref{section:main}. We note, without proof, that the wreath product of a finite abelian group
and a free abelian group, is residually nilpotent. While harder to prove, the wreath product of two torsion-free abelian groups is residually nilpotent as well. We refer the reader to the papers by  Gruenberg~\cite{Gruenberg}, Lichtman~\cite{Lichtman} and
Hartley \cite{Hartley} where proofs can be found, or where the results described there can be used to prove them.

We are now in a position to prove the various parts of Theorem 9.1.  In the remainder of this section we will use the term {\em direct product} and the product notation to denote the previously discussed restricted product.

\begin{proof} We prove these statements sequentially.

\begin{enumerate}
\item Let $W=A \wr T$ be the wreath product of a cyclic group, $A$, of prime order $p$ generated by $a \in A$ and an infinite cyclic group, $T,$ generated by $t$.
We have already noted that $W$ is residually nilpotent. Our objective now is to prove that
any finitely generated, residually nilpotent,  metabelian group $H$
with the same lower central sequences as $W$ is isomorphic to $W$.
 The normal closure  $B$ of $a$ in $W$ is  the direct product of its subgroups
$gp(a_i)$, $i=1,2,\dots$, where $a_i=t^{-i}at^i$ is order $p$. Since each
of the $a_i$ is of order $p$
we can view $B$ as a module over the group ring  $\Z_p[T]$ of $T$ over the field
$\Z_p$ of of $p$ elements. Since the $a_i$ generate their direct product, $B$ is
a free $\Z_p[T]$-module.

 Since $H$ has the same lower central
sequences as $W$, $H/[H,H]$ is the direct product of an infinite cyclic
group on $u[H,H]$ and a group of order $p$ generated by $b[H,H]$. There is, for each $c$,
an isomorphism $\phi_{c+1}$ mapping  $H/\gamma_{c+1}(H)$ onto $W/\gamma_{c+1}(W)$.
This gives rise to a monomorphism $\phi$ between the respective  direct products
\[
\prod_{c=1}^{\infty}H/\gamma_{c+1}(H) \longrightarrow
\prod_{c=1}^{\infty}W/\gamma_{c+1}(W).
\]
Since $H$ is residually nilpotent, this monomorphism induces a monomorphism of $H$ into $\prod_{c=1}^{\infty}W/\gamma_{c+1}W$.
It follows that this induces a monomorphism of $[H,H]$ into  $\prod_{c=1}^{\infty}W/\gamma_{c+1}W$, where the finite order elements of the latter group have order $p$. So $[H,H]$ is an abelian group
of exponent $p$.

Now adjoin $b$ to $[H,H]$. One easily checks that the only torsion in $\prod W/\gamma_{c+1}(W)$ has order $p$.  Since $b$ is of finite order modulo $[H,H]$ and since $A$ is abelian of exponent $p$,it follows that $b$ has finite order, and since the above homomorphism is 1-1, $b$ has finite order $p$. It follows from the residual nilpotence of $H$ that $b$ commutes
with all of $[H,H]$. Otherwise there exist an element $h\in [H,H]$ such that
$gp(b,h)$ is not abelian. However $gp(b,h)$ is then a finite subgroup of $H$ and so is isomorphic
to a finite subgroup of $\prod_{c=1}^{\infty}W/\gamma_{c+1}W$. Now the torsion subgroup of $W$
can be expressed as an ascending union
of abelian groups of exponent $p$ so $gp(b,h)$ embeds into one of these subgroups in which
the images of $b$ and $h$ are independent, and hence, commute. This implies that $K = gp(b, [H,H])$
is a normal abelian subgroup of $H$ of exponent $p$.

So $K$ is a module over $\Z_p[H/K] \cong \Z_p[T]$.  Since this ring is a p.i.d., $K$ is a sum of cyclic modules.  In the event that the number of these cyclic summands is at least 2, then $H/[H,H]$
will be the direct product
of an infinite cyclic group and at least two groups of order $p$, which is not the case.

So there is only one cyclic submodule generated by an element, say $h$. Since the submodule of $K$
generated by $h$ has infinite rank as a vector space over $\Z_p$, it follows that it is isomorphic to
$ \Z_p[T]$. So $H \cong gp(h)\wr gp(s)$, where $sK$ generates $H/K$,
and so is isomorphic to $W$. This completes the proof of (1).

\item As above, we view $B$ as a module over $\Z_p[T]$.
$B$ is a free $\Z_p[T]$-module on a single element $a$. We start by finding a
preferred presentation of $W$~\cite[Page 185]{Lennox-Robinson}. Our objective
is to show that there is an algorithm which recursively enumerates
all preferred presentations of groups
isomorphic to $W$ and recursively enumerates all preferred presentations of groups which are not
isomorphic to $W$.

To this end let $V$ be a finitely generated metabelian group
defined by a given preferred presentation. At the outset we algorithmically check that $V$
is residually nilpotent~\cite[Cor. 9.2]{BCR}, otherwise the algorithm terminates.  There is an algorithm to find a finite module presentation for $[V, V]$,~\cite[Thm. 3.1]{BCR}, and since the word problem is solvable for the finitely generated metabelian group $V$, one can algorithmically determine if $[V, V]$ is trivial.  If $[V,V]$ is trivial, then $V \ncong W$ and the algorithm halts. Otherwise, we can find a non-trivial element $v\in [V,V]$. There is then an algorithm to compute a  preferred  presentation of  the centralizer $C$ of $[V,V]$ ~\cite[Thm. 6.1]{BCR}.  The group $C/[V,V]$ is a subgroup of $V/[V,V]$ and is therefor a finitely generated abelian group.  We can algorithmically check if it  is cyclic of order $p$. If not, then $V \ncong W$. So we assume $C/[V,V]$ is finite or order $p$.  Since $C$ if generated by $[V,V]$ and just one more generator, and since $C$ centralizes $[V,V]$ with a single additional generator, $C$ must be abelian.

Since $[V,V]\subset C$, $V/C$ is a finitely generated abelian group.  So we can also check algorithmically whether it is infinite cyclic.  If not, $V \ncong W$.  Otherwise, $V/C$ is infinite cyclic.  Since $C$ is a finitely generated module over $V/C$ and the word problem is is solved for finitely generated metabelian groups, we can algorithmically check to see if each module generator has order $p$, and therefor whether $C$ is abelian of  exponent $p$.  If not, $V \ncong W$.  If so, $C$ is a module over the mod-$p$ group ring $\Z_p[V/C] \cong \Z_p[T]$, a p.i.d.  Thus, $C$ is a sum of cyclic modules.  If $C$ has more than one summand, then as in part~$(1)$, $V/[V,V]$ has more than one cyclic summand of order $p$, and $V \ncong W$.  Otherwise, $C$ has only one summand as a $\Z_p[T]$-module.  Since $C$ has infinite rank as a vector space over $\Z_p$, $C$ is a free rank one module over $\Z_p[T]$.

We have shown that $V$ is generated by
an element $s$ which is of infinite
order modulo $C$, that $C$ is a free rank one $\Z_p[V/C]$-module on an element $b$.  Since the quotient homomorphism $V \to V/C$ splits with kernel $C$, it follows that $V \cong W$.

\item  Let $W=A \wr T$, where $A$ is the infinite cyclic group on $a$ and
$T$ is the infinite cyclic group on
$t$. The normal closure  $B$ of $a$ in $W$ is
freely generated by the conjugates $a^{t^i}$ of $a$ by the
powers of $t$. So if we view $B$ as a module over the integral group
ring $\Z[T]$ of $T$, then
$B$ is a free module on $a$. We start out then by finding a
preferred presentation of $W$. Our objective
is to show that there is an algorithm which recursively
enumerates all preferred presentations of groups
isomorphic to $W$ and recursively enumerates all preferred presentations of groups which are not
isomorphic to $W$.

To this end let $V$ be a
 group defined by a given preferred presentation. As in the last argument, there is an algorithm which computes  a module presentation of the
derived group of $V$  and a second algorithm
that computes a presentation for the centralizer $C$
of the derived group of $V$. As in the prior argument, another algorithm determines whether of not
$C$ is abelian. If $C$ is not abelian,
then $V$ is not isomorphic to $W$ and the algorithm comes to a halt.
So we suppose that $C$ is abelian.
This means that $C$ is an abelian normal
subgroup of $V$ containing the derived group of $V$. We can now
determine whether of not $V/C$ is infinite
cyclic. If it is not, then again $V \ncong W$ and the algorithm comes to an end.  Suppose then that $V/C$ is infinite
cyclic. View $C$ as
a module over $V/C$. There is now an algorithm to decide whether or not $C$ is projective. If it is not, then
$V \ncong W$. However if $V$ is projective then
by the Quillen-Suslin theorem~\cite{suslin},  $C$ is a free
module over the infinite cyclic group. If the rank of $C$ is
different from 1, then $V\ncong W$. On the other
hand if $C$ is free of rank 1, then $V \cong W$ and the algorithm terminates.
\item  Let $W=A \wr T$, where $A$ is infinite cyclic on $a$ and $T$ is infinite cyclic on $t$. Then
$W$ is residually nilpotent.  Put $H=gp((a^2)^t a^{-1},\   a^2a^{t^{-1}},\ t)$.
Then $(a^2)^ta^{-1}, a^2a^{t^{-1}}, t$
generate $W$ modulo $\gamma_2(W)$ and hence they generate $W$ modulo $\gamma_{c+1}(W)$ for every
$c$, i.e., $H\gamma_{c+1}(W)=W$ for every $c$. Consequently
\[
W/\gamma_{c+1}(W)=H\gamma_{c+1}(W)/\gamma_{c+1}(W)\cong H/H\cap \gamma_{c+1}(W).
\]
Since $\gamma_{c+1}(W)\geq \gamma_{c+1}(H)$ in view of the fact that
finitely generated nilpotent groups
are Hopfian, $H\cap \gamma_{c+1}(W)=\gamma_{c+1}(H)$. So we have proved that $H$ and $W$
have the same lower central sequences. As a consequence $W$ is para-$H$ since now it follows that
the inclusion of $H$ into $W$ induces isomorphisms between the factor groups $H/\gamma_{c+1}(H)$
and $W/\gamma_{c+1}(W)$. Thus, by the Telescope Theorem, $H$ and $W$ are para-equivalent.

However $H$ and $W$ are not isomorphic. To see that this is so,
suppose the contrary. Any isomorphism
from $H$ to $W$ will map the centralizer $C$ of any non-trivial element of $[H,H]$ isomorphically
onto the centralizer $D$ of a non-trivial element of $[W,W]$. If we view $C$ as a $\Z[T]$-module
and $D$ similarly also as a $\Z[T]$-module, then these modules must be isomorphic.
But $D$ is a cyclic module
and it is not hard to prove that $C$ is a two-generator module since the ideal
of $\Z[T]$ generated by $2t-1,2-t$ is not  a principal ideal of $\Z[T]$, i.e., the class number
of $\Z[T]$ is at least 2.
\end{enumerate}
\end{proof}

It is this approach that we will take in our third paper which takes advantage of class field
theory to construct a number of interesting examples of residually nilpotent, metabelian,
polycyclic groups whose para-equivalence classes are not singletons.

\subsection{Polycyclic metabelian groups with para-equivalence classes that are not singletons}\label{section:singletons}

Recall from the Telescope Theorem, Theorem 5.6,  that if $G \to H$ is a para-equivalence of residually nilpotent polycyclic metabelian
groups, then $G$ is isomorphic to a subgroup of finite index in $H$ and
$H$ is isomorphic to a subgroup of finite index in $G$.
This does not imply $G \cong H$ as the following theorem demonstrates.

We note that Grunewald and Segal proved the Isomorphism theorem in the special case of finitely generated nilpotent groups, and Segal completed the proof for polycyclic groups. See the complete proof of the solution of the isomorphism problem, which contains the joint work of Segal and Grunewald followed by that of Segal, in the book~\cite{Segal2}.

\begin{theorem}\label{thm:eg-polycyclic}
There exist finitely generated, residually nilpotent, para-equivalent polycyclic
meta-belian groups which are not isomorphic.

\end{theorem}

\begin{remark}
Ideal class theory inspired this example.  It's well known that
the ideal class group of $\Q(\zeta_{23})$ has order $3$ and is
generated by the non-principal ideal $(2, 1 + P) \subset
\Z[\zeta_{23}]$, where $P$ is the Gaussian period described in the
proof below.  (See, for instance,~\cite[page 86]{Marcus}.)
\end{remark}

\begin{proof}
Let $T$ be the infinite cyclic group on $t$. Consider the Dedekind domain
\[
\Z[\zeta_{23}] \cong \Z[T]/(N(t)) \text{ where } N(t) = \Sigma_{k=0}^{22} t^{k}.
\]
We view $\Z[\zeta_{23}]$ as a $\Z[T]$-module with $t$ acting on
$\Z[\zeta_{23}]$  by multiplication by $\zeta_{23}$. Let $G$ be the semi-direct
product of $\Z[\zeta_{23}]$ by $T$ using this action:
\[
G = \Z[\zeta_{23}] \rtimes T
\]

Observe that the augmentation $\epsilon$ from $\Z[T]$ onto $\Z$
 determines a commutative diagram
where $C_{23}$ denotes the cyclic group with $23$ elements.
\[
\begin{diagram}
\node{\Z[T]} \arrow{e,t}{\epsilon} \arrow{s,r}{q}\node{\Z}\arrow{s}\\
\node{\Z[\zeta_{23}]} \arrow{e} \node{C_{23}}
\end{diagram}
\]
Therefore if $p(t) \in \Z[T]$ and $p(1) = \pm 1$, then
$q(p(t)) \neq 0$.  Hence, since $\Z[\zeta_{23}]$ is an integral
domain, the set $S \subset \Z[\zeta_{23}]$, that is the
image of $1 + \ker \{\epsilon \colon \Z[T] \to \Z\}$ in
$\Z[\zeta_{23}]$, is a multiplicative set in $\Z[\zeta_{23}]$.  One easily checks that $G$ is
polycyclic, metabelian and residually nilpotent as well.

With the above observations in mind, we now construct a residually nilpotent group $H$, and a para-equivalence
$G \to H$ such that $G$ and $H$ are not isomorphic.

Consider the {\em Gaussian period}
\[
P = \Sigma_{k=1}^{11}\zeta_{23}^{k^{2}} = \zeta_{23} + \zeta_{23}^{2} + \zeta_{23}^{3} + \zeta_{23}^{4} + \zeta_{23}^{6} + \zeta_{23}^{8} + \zeta_{23}^{9} + \zeta_{23}^{12} + \zeta_{23}^{13} + \zeta_{23}^{16} + \zeta_{23}^{18} = \frac{-1 + \sqrt{-23}}{2},
\]
and the element
\[
1 + P = \frac{1 + \sqrt{-23}}{2} \subset \Z[\zeta_{23}].
\]
Consider the non-principal ideal $I$ of $\Z[\zeta_{23}]$ generated
by $2$ and $1+P$, and let
 $H$ be the semi-direct product of $I$ and  $T$ with $t$ acting on $I$
by multiplication by $\zeta_{23}$:
\[
H = I\rtimes T.
\]
(We leave it to the reader to check that this ideal is non-principal.) Observe that the inclusion of $I$ in $\Z[\zeta_{23}]$
induces an inclusion of  $H$ into  $G$.

We now construct an element $s \in S$ such that
the principle ideal of $\Z[\zeta_{23}]$ generated by $s$
is properly contained in $I$

Let
\[
p(t) =  2(1 + t + t^{2} + t^{3} + t^{4} + t^{6} + t^{8} + t^{9} + t^{12} + t^{13} + t^{16} + t^{18}) - N(t).
\]
Then
\[
p(t) \in 1 + I,
\]
since $p(1) = 1$.  Now let $s = p(\zeta_{23}) \in S$. Also,
\[
p(\zeta_{23}) = 2 + 2P = 1 + \sqrt{-23} \in \Z[\zeta_{23}].
\]

Now consider the following diagram where $\alpha$ is the homomorphism
such that $\alpha(1) = s$, that is $\alpha(1) = 2(1+P)$:
\[
\begin{diagram}
\node{\Z[\zeta_{23}]}\arrow{e,t}{\alpha}\node{(2, 1+P)}\arrow{e,t}{\subset}\node{\Z[\zeta_{23}]}\arrow{e,t}{\alpha}\node{(2, 1+P)}
\end{diagram}
\]
Each composition of two homomorphisms is given by multiplication
by $s \in S$.  So all the above inclusions are isomorphisms after inverting the multiplicative set $S$.  By Lemma~\ref{Lemma:structure thm}(5), all homomorphisms in this diagram induce
isomorphisms on I-adic quotients.

Therefore the homomorphisms in the corresponding diagram of groups induces isomorphisms on lower central series quotients,
\[
G \xrightarrow{\alpha \rtimes id}  H \subset G \xrightarrow{\alpha \times id} H
\]
and
\[
G \xrightarrow{\alpha \rtimes id} H
\]
is a para-equivalence.

However
\[
G = \Z[\zeta_{23}] \rtimes T \not\cong I \rtimes T = H,
\]
since $I$ is not a principal ideal in $\Z[\zeta_{23}]$,
 and therefore not isomorphic to $\Z[\zeta_{23}]$.
\end{proof}

\section{The groups $G_n=\langle a,t; a^t= a^n \rangle$}\label{section:gn}
We shall prove  that the $G_n$, $n \in \Z$, have a number of interesting properties, which we have collected together
in the following theorem.

\begin{theorem}\label{thm:Gn}
Consider the group $G_n$, $n \neq 2$,  defined as above.
\begin{enumerate}
\item If $G_n$ is residually nilpotent.
\item If $H$ is a residually nilpotent group with the same lower central sequences as $G_n$, then
$H \cong G_n$.
\item The Isomorphism Problem is solvable for $G_n$.
\end{enumerate}
\end{theorem}

\begin{proof} We prove these statements sequentially.
\begin{enumerate}
\item We will restrict our attention to the case where $n>2$ since the remaining
cases can be taken care of in much the same way.

Put $G=G_n$. Notice first that $a$ is of infinite order
since by a theorem of Magnus, Karrass and Solitar, $a^ta^{-n}$  is not a proper power
in the free group on $a$ and $t$~\cite{MKS}. Let $A$ be the normal closure in $G$ of $a$ and put
$a_{j}=t^jat^{-j}$ and $A_j=gp(a_j)$. Since $a^t=a^n$, it follows that $a_{j+1}^n=a_j$. So $A$
is generated by $a=a_0,a_1,\dots$ and is an ascending union of infinite cyclic groups.
Consequently $G$ is the semidirect product of the torsion-free abelian group $A$ and an infinite cyclic
group and is therefore metabelian.
 Now $[a,t]=a^{n-1}$ and hence
$\gamma_2(G)$ is the normal closure in $G$ of $a^{n-1}$ and $G/\gamma_2(G)=C_{\infty}\times C_{n-1}$, where
$C_{\infty}$ is an infinite cyclic group and $C_{n-1}$ is a cyclic group of order $n-1$.

If we now put $b_j=[a,\underbrace{t,\dots,t}_{j}]$ for $j=2, \dots$,
and $B_j=gp(b_j)$, then $b_{j+1}^{n-1}=b_j$ and $\gamma_{j}(G)=gp(B_j,B_{j+1},\dots)$. It follows that
$G/\gamma_{j+1}G= C_{\infty}\ltimes C_{(n-1)^j}$ and that $\gamma_{\omega}(G)=1$ and therefore $G$ is residually
nilpotent.\\

\item Suppose then that $H$ is a finitely generated, residually
nilpotent group with the same lower central sequences as $G$. Then
$H$ embeds into the unrestricted direct product $\tilde{H}$ of the
factor groups $H/\gamma_k(H)$. Consequently $H$ is isomorphic to a
subgroup of a direct product of metabelian groups and is therefore
metabelian.

By hypothesis, $H/\gamma_k(H) \cong G/\gamma_k(G)$. Let  $\phi_k$ be a homomorphism
mapping $H/\gamma_k(H)$ onto $G/\gamma_k(G)$ for each $k$. Choose $\tau \in H$ to be any element
of $H$ such that the image  of $\tau \gamma_3(H)$ under $\phi_3$ is $t\gamma_3(G)$ and choose
$\alpha \in H$ such that the image of $\alpha$ under $\phi_3$ is $a \gamma_3(G)$. Denote by
 $\tau_k$ the image of $\tau$ under $\phi_k$ and by $\alpha_k$ the image of $\alpha$ under
$\phi_k$. It follows that
$$\alpha_k^{\tau_k} \cong \alpha_k^n\,  \rm{modulo}\, \gamma_k(H),$$
for every $k$, i.e.,
$$(\alpha)^{\tau}(\alpha)^{-n} \in \gamma_k(H)$$
for every $k>2$. But the intersection of the $\gamma_k(H)$ is trivial. Therefore
$\alpha^{\tau}=\alpha^n$.

Since $G$ is torsion-free, so too is $\widehat{G}$, and using the
$\phi_k$ we can identify $H$ with a subgroup of $\widehat{G}$. It
follows that every abelian subgroup of $H$ is torsion-free.
 Since $H$ is finitely generated
we can supplement $(\alpha)^{n-1}=c_1$ with finitely many elements $c_2,\dots,c_{\ell}$
of $[H,H]$ which together with $\tau$ and $\alpha$
suffice to generate $H$. Now choose finitely many elements $d_1,\dots,d_j$ in $[H,H]$
which freely generate the subgroup of $[H,H]$ generated by $c_1,\dots,c_{\ell}$.  Since $H$ is residually nilpotent and has the same lower central sequence as $G$, it follows that  $\tau$ conjugates each of $d_1,\dots,d_j$ to their $n^{th}$ powers. It follows then from this
discussion that $H$ is generated by $\tau, \alpha, d_1,\dots,d_k$ and defined
by the relations which specify that the elements $\alpha,d_1,\dots,d_j$ commute,
that $\tau$ conjugates each of the $d_j$ into their $n^{th}$-powers and that
$\alpha^{\tau}=\alpha^n$. This implies that $H/[H,H]$ is a direct product
of an infinite cyclic group and $j+1$ cyclic groups of order $n-1$.  This implies that $j = 0$ and so $H \cong G$.\\

\item  Let $H$ now denote a finitely generated metabelian group
given by a preferred presentation.  We have to prove that there is an algorithm which
determines whether or not $H\cong G$, which we describe in stages, as in the proof above.

There is an algorithm that decides whether or not $H$ is
residually nilpotent. If not, then $H \ncong G$. Suppose  then that $H$ is residually
nilpotent. There is an algorithm that decides whether $H$ is torsion-free and also
whether $H$ is not abelian.  If $H$ is abelian or if $H$ contains a non-trivial element
of finite order, then $H \ncong G$. So we can proceed under the assumptions that $H$ is
not abelian and that $H$ is torsion-free. We now check to see if $H/[H,H]$ is the direct
product of an infinite cyclic group and a cyclic group of order $n-1$. If this is not the
case then $H\ncong G$. So we can assume that $H/[H,H]$ is such a direct product. Now choose
a non-trivial element $x \in [H,H]$. Then there is an algorithm to compute the centralizer
$C$ of $x$. We now check to see whether $C$ is abelian and whether $H/C$ is infinite
cyclic. If not, then $H \ncong G$. Suppose then that $C$ is abelian and $H/C$ is infinite
cyclic generated by, say, $xC$. We now view $C$ as a module over the integral group ring $R$ of
the infinite cyclic group generated by $x$. The $R$-module $C$ is a finitely generated
$R$-module. So we can find a finite set of elements $d_1,\dots, d_q$ of $C$ which generate
$C$ as an $R$-module.

If $C/[H,H]$ is not cyclic of order $n-1$ then $H \ncong G$. Let us
assume that $C/[H,H]$ is cyclic of order $n-1$. Choose an element $y\in C$ such that
$y[H,H]$ generates $C/[H,H]$. Then $y$ has order $n-1$ modulo $[H,H]$. Now there is
an algorithm to decide if the $R$
submodule $K$ of $[H,H]$ generated by $z=y^{n-1}$ is equal to $[H,H]$,~\cite[Cor. 5.3]{BCR}. If it is
not, then $H \ncong G$. Suppose then that $K=[H,H]$.  It follows that $H$ is generated
by $x$ and $y$, that $y^x=y^n$ and $K$ is an ascending union of infinite cyclic groups,
as needed.
\end{enumerate}
\end{proof}

\bibliographystyle{amsalpha}

\end{document}